\documentclass[final,arxiv]{elsarticle}

\makeatletter
\@ifclasswith{elsarticle}{arxiv}{
	\def\ps@pprintTitle{%
		\let\@oddhead\@empty
		\let\@evenhead\@empty
		\def\@oddfoot{\centerline{\thepage}}%
		\let\@evenfoot\@oddfoot}
}{}
\makeatother

\usepackage{amsmath,amssymb,amsfonts,amsthm}
\usepackage{mathtools}
\usepackage{bbm}
\usepackage{algorithm}
\usepackage{algpseudocode}
\usepackage{hyperref}
\usepackage{pgfplots}
\usepackage{booktabs}
\usepackage{pgfplotstable}

\pgfplotstableset{
	every head row/.style={before row=\toprule,after row=\midrule},
	clear infinite
}

\usepackage{tikz}
\usetikzlibrary{calc,arrows,arrows.meta,positioning,shapes,snakes,
	intersections,automata,petri,backgrounds,patterns,matrix,external}

\usepackage{shellesc}

\DeclarePairedDelimiter{\norm}{\lVert}{\rVert}

\newtheorem{lemma}{Lemma}[section]
\newtheorem{theorem}[lemma]{Theorem}
\theoremstyle{remark}
\newtheorem{remark}[lemma]{Remark}
\theoremstyle{definition}

\usepackage{color}
\usepackage{cleveref}

\newcommand{\mtta}{\mathrm{MTTA}}

\begin{document}
	
	\begin{frontmatter}
		\author[ISTI]{Giulio Masetti}
		\ead{giulio.masetti@isti.cnr.it}
		
		\author[DM,ISTI]{Leonardo Robol\corref{cor1}\fnref{gncs}}
		\ead{leonardo.robol@unipi.it}
		
		\address[ISTI]{Institute of Information Science and
			Technologies ``A. Faedo'', ISTI-CNR, Pisa, Italy.}
		\address[DM]{Dipartimento di Matematica, Universit\`a di Pisa, Italy.}
		
		\cortext[cor1]{Corresponding author}
		
		\fntext[gncs]{The research of the author was partially supported by the INdAM/GNCS project ``Metodi di proiezione per equazioni di matrici e sistemi lineari con operatori definiti tramite somme di prodotti di Kronecker, e soluzioni con struttura di rango''. The author is a member of the INdAM Research group GNCS.}
		
		\title{Tensor methods for the computation of MTTA in large
systems of loosely interconnected components}		
		
\begin{abstract}
We are concerned with
the computation of the mean-time-to-absorption (MTTA) for 
a large system of loosely interconnected components, modeled 
as  continuous time Markov chains. 
In particular, we show that splitting the local and synchronization
transitions of the smaller subsystems allows to formulate an
algorithm for the computation of the MTTA which is proven to
be linearly convergent.  Then, we show how to modify the method
to make it quadratically convergent, thus overcoming the difficulties
for problems with convergent rate close to $1$. 

In addition, it is shown that this decoupling of local and synchronization transitions allows to easily represent all the matrices and vectors
involved in the method in the tensor-train (TT) format --- and 
we provide numerical evidence showing that this allows to treat
large problems with up to billions of states --- which would otherwise
be unfeasible. 
\end{abstract}
		
	\end{frontmatter}

\section{Introduction}

Model-based analysis of large and complex systems is 
considered of fundamental importance to tackle 
the increasing complexity of our society; at the same time, 
it is challenging,
because of the great number of technical issues that need to be 
overcome in order to accomplish the task.

On the most prolific areas in stochastic modeling is 
represented by Markov chains  
because of the trade-off
between representativeness and availability of solution techniques
that offers to the modeling community.
Nevertheless, when the number of interacting components 
in a system increases, even solution techniques that
are known to scale well can suffer for the largeness problem.
Of particular interest for this paper is the case when the modeler 
is asked to deal with reward structures on
Continuous Time Markov Chains (CTMCs)~\cite{TB17},
where the model is represented by an infinitesimal generator 
matrix $Q$, an initial probability vector $\pi_0$ and
a reward vector $r$. Standard approaches to tackle model largeness,
e.g., lumping and symbolic representation~\cite{D93} of $Q$, might not
be sufficient to address extremely large system models.
The symbolic representation of $Q$ by itself can in fact reduce 
the storage of information related to the chain, 
but cannot reduce the memory
footprint of the vectors involved in the computations. 

An example can be obtained by combining smaller subsystems into a 
larger one; if we suppose to consider $k$ components with only two possible
states (for instance, working and failed), which are combined in a single
Markov chain, the state space could be as large as $2^k$, including all
the possible combination of states in the components. For instance, 
if $k = 50$, we might have up to 
$2^{50}$ chain states; storing every single vector,
whose entry are assumed to occupy $8$ bytes, would require
more than $8$ Petabyte, making unfeasible
to store it in RAM. 

To overcome this problem, one may consider 
a clever use of a combination of RAM
and disk storage or considering approximations of the 
chain~\cite{BK04}; nevertheless, the gain is still 
not completely satisfactory, and indeed this
is considered one of the main obstacle to
the scalability of analytical methods for the evaluation
of CTMC properties.
  
Thus, a symbolic representation also of the vectors involved in
the computations can be considered a break-trough for the
modeling community to be able to analyze chains with huge state space.
Two recent papers have pursued this direction for
\emph{irreducible chains}, where the steady-state probability
vector $\pi$ is computed. In \cite{kressner2014low}, this is
achieved exploiting 
tensor trains~\cite{oseledets2011tensor}, whereas 
in \cite{BDKO17} the authors use Hierarchical Tucker Decomposition~\cite{grasedyck2010hierarchical,kressner2014htucker}.

The focus of this paper, instead, is on performance, dependability and 
performability properties of CTMC with \emph{absorbing states}. A new
symbolic representation of both matrices and vectors is proposed
to enable the assessment of huge models.
In particular, a first step forward with respect to the available
numerical techniques will be detailed for the case of
the Mean Time To Absorption (MTTA) evaluation in the context 
of reliability modeling~\cite{TB17}, where the CTMC has an unique
absorbing state and we are interested in computing the $\mtta$ as
a cumulative reward measure.
This measure sheds light on limit behaviors of the chain\footnote{the $\mtta$ is the limit of the integral of a function of the probability vector $\pi(t)$, for $t\rightarrow\infty$, as detailed in~\Cref{eq:defMTTA}.} but does not describe a steady-state
property of the chain.

The new method is presented formally and applied to a simple but
representative case study, where the technique is proved to be
feasible for CTMC with up to $3^{32} \approx 10^{15}$ 
(potential) states, and is supposed to scale even further. Storing
a vector of double floating point numbers of this length would 
require more than $13$ PetaByte of memory. 
A recent application of the method we are going to present can be found in~\cite{MR19}.
The paper is structured as follows:
In~\Cref{sec:SAN} notation for the Stochastic Automata Network (SAN)
formalism will be recalled.
\Cref{sec:lowRank} introduces the new low-rank representation for 
both the matrices and vectors.
\Cref{sec:MTTA} specializes the new general method to
the evaluation of the $\mtta$, providing all the details of the
mathematical steps.
In~\Cref{sec:remarks} very important computational remarks are
discussed: the feasibility of the method strongly relies on the
application of few key steps, detailed in this section.
\Cref{sec:caseStudy} presents the case study and
in~\Cref{sec:results} numerical results are show the feasibility 
of the method.
Finally, in~\Cref{sec:conclusions} conclusions are drawn.

\section{Stochastic Automata Networks}
\label{sec:SAN}

As described in~\cite{BK04}, it is possible to address the study of
large CTMC defining symbolically 
the infinitesimal generator matrix $Q$,
thus avoiding a complete state-space exploration. 
In particular, the Stochastic Automata Network (SAN) formalism~\cite{Plateau2000} 
allows to represent the CTMC as a set of stochastic automata
$M_1,\dots,M_k$, each having a (small) reachable set of states $\mathcal{RS}_i$,
where transitions are of two kinds: \emph{local} and
of \emph{synchronization}.
Transitions $t$ that are local to $M_i$, written $t\in\mathcal{LT}_i$, 
have impact only on $\mathcal{RS}_i$
and indicate the switch from a state $s\in\mathcal{RS}_i$ to
a state $s^{\prime}\in\mathcal{RS}_i$, in the following written
$s\overset{t}{\rightarrow}s^{\prime}$.
Synchronization transitions $t\in\mathcal{ST}$, instead, can appear in
more than one automaton.
In particular, if $t\in M_{i_1}$, $t\in M_{i_2}$, \dots, 
and $t\in M_{i_h}$ then
$s_{i_j}\overset{t}{\rightarrow}s^{\prime}_{i_j}$ can fire only if
the automaton $M_{i_1}$ is in state $s_{i_1}$, and 
the automaton $M_{i_2}$ is in state $s_{i_2}$, \dots,
and the automaton $M_{i_h}$ is in state $s_{i_h}$, at the same time.    
The overall CTMC is then the orchestration of local stochastic automata
where the director is represented by $\mathcal{ST}$.
The infinitesimal generator matrix $Q$ is not assembled explicitly, 
and its compressed
representation is called \emph{descriptor matrix} and is
formally defined by 
\begin{equation}
\label{eq:descriptorQ}
Q = R + W + \Delta,
\end{equation} 
i.e., the sum of local contributions, called $R$, and
synchronization contributions, called $W$, where 
\begin{align}
R &= \bigoplus_{i=1}^{k} R^{(i)}, &
W &= \sum_{t_j\in\mathcal{ST}} \bigotimes_{i=1}^{k} W^{(t_j,i)},
\end{align}
$R^{(i)}$ and $W^{(t_j,i)}$ are
$|\mathcal{RS}^{(i)}|\times|\mathcal{RS}^{(i)}|$ matrices,
and the diagonal matrix $\Delta$ is defined as
$\Delta = -\text{diag}\left( (R+W) e \right)$ where
$e$ is the vector with all the entries equal to $1$. The 
operator $\oplus$ is the \emph{Kronecker sum}, as 
formally described in Section~\ref{sec:kroneckersum}.
The matrices $R^{(i)}$ and $W^{(t_j,i)}$ are assembled exploring
$\mathcal{RS}^{(i)}$ and can be specified through an high level
formalism such as GSPN~\cite{D93,CM99} or PEPA~\cite{H96}.
In particular, $W^{(t_j,i)} = \lambda_{t_j} \tilde W^{(t_j,i)}$
where $\tilde W^{(t_j,i)}$ is a $\{0,1\}$-matrix defined as follows:
\begin{equation}
\tilde W^{(t_j, i)}_{s_i, s_i'} =
\begin{cases}
1 & \text{if } t_j \text{ is enabled in } s_i \text{ inside } M_i \text{ and } s_i\overset{t_j}{\rightarrow}s_i^{\prime} \\ 
0 & \text{otherwise} \\
\end{cases}
\end{equation}
where $\lambda_{t_j}$ is the constant rate associated with $t_j$, 
equal in every $M_i$. In particular, if the transition $t_j$ has 
no effect on the component $M_i$, we have $\tilde W^{(t_j, i)} = I$.
In the following we will call 
\[
\mathcal{PS} = \mathcal{RS}^{(1)}\times \cdots \times \mathcal{RS}^{(k)}
\]
the \emph{potential state space} and the $|\mathcal{PS}|\times |\mathcal{PS}|$
descriptor matrix $Q$ will be treated implicitly. 

\begin{remark}
	In general, the set of reachable states given an initial
	probability distribution $\pi_0$, might be a strict subset
	of $\mathcal{PS}$. Some techniques exploit the fact that
	the reachable state is smaller to achieve a higher efficiency. However, 
	this hides the tensorized structure of 
	$Q$ and of the probability vectors to be computed; 
	therefore, in this work the focus will be on $\mathcal{PS}$. 
\end{remark}

\section{Low-rank tensors}
\label{sec:lowRank}

To overcome the exponential explosion of memory requirements, 
sometimes called \emph{curse-of-dimensionality} \cite{oseledets2011tensor}, 
there has been a recent trend in exploiting the structure of $Q$, 
which can be recognized from \Cref{eq:descriptorQ}, in the setting where
the model describes the interaction of loosely interconnected component; in fact, in this case $Q$ can
be efficiently stored by only memorizing factors of Kronecker products \cite{BK04}. This
enables a reduction in storage and an acceleration of the
matrix-vector operator required in the development of 
most of the algorithms for computing steady-state probabilities
and performance and reliability measures. 

However, with the exponential growth of the state space which can
happen when combining several systems, even storing vectors
with as many components as the cardinality of the (potential) state space
can quickly become unfeasible. For this reason, there has been recently
a shift in developing ``symbolic'' representations for the
vectors under consideration as well. This turns out to require
considerably more effort. Recent promising developments leverage
the use of low-rank tensor formats, namely hierarchical Tucker
decompositions \cite{BDKO17}, and Tensor Trains \cite{kressner2014low} (in this work
considered for the steady-state analysis).

For the sake of self-completeness, we briefly review the theory of low-rank
tensor operators and Tensor Trains, 
that will be the building block for the compressed representation
proposed in this work. We refer the interested reader to \cite{oseledets2011tensor} 
for further details.  

\subsection{Kronecker sums} \label{sec:kroneckersum}

As we discussed in Section~\ref{sec:MTTA}, there are some kind of structures
that appear in the definition of the infinitesimal generator $Q$. In particular, 
we may define the \emph{Kronecker sum} $\mathcal A := A_1 \oplus \ldots \oplus A_k$ as follows:
\[
    \mathcal A = A_1 \otimes I \otimes \ldots \otimes I + 
    I \otimes A_2 \otimes I \otimes \ldots \otimes I + 
    \ldots + I \otimes \ldots \otimes I \otimes A_k. 
\]
Similarly structured matrix arise in other contexts as well, such as the discretization of high-dimensional PDEs, and solution of matrix and
tensor equations. 

 Equation~\ref{eq:descriptorQ} shows that this is the form
 of the $R$ matrix in the definition of the infinitesimal generator $Q$ of
the Markov chains under consideration. The matrix $Q$ is then obtained 
adding $W$, which models the weak interaction between the components, 
and has a similar structure. Even if $Q$ is not exactly in the
form of a Kronecker sum, but it has a low-rank tensorial structure. 

The term ``tensor rank'' does not have a single universally
accepted meaning. In fact, unlike in the matrix case (which
is obtained by setting $k = 2$), 
several different ranks can be defined --- and they have different
computational properties. The most classical definition is the
CP rank, linked to the Canonical Polyadic Decomposition (also sometimes
called PARAFAC --- see \cite{kolda2009tensor} and the references therein
for more details); this is linked to the definition of rank as sum of 
rank $1$ terms, which are in turn defined as outer
product $v_1 \otimes \dots \otimes v_k$. However, using this 
low-rank format is inherently difficult and unstable. For instance, 
the set of rank $R$ tensors is not closed if $k > 2$, and this 
makes the low-rank approximation problem ill-posed \cite{desilva2008tensor}. Moreover, the computation of the best rank $r$ approximation of a 
tensor is a difficult (indeed, NP-hard, \cite{hillar2013most}), and the solution can only be approximated
by carefully adapted optimization algorithms, see \cite{kressner2014low}
for a review. 

For this reason, there has been interest in finding alternative low-rank
representation of tensors. A very robust possibility that is well-understood is the Tucker decomposition, linked to the Higher
Order SVD (HOSVD) \cite{delathauwer2000multilinear}. However, this
approach requires to store a $k$-dimensional tensor (even though
of smaller sizes), and so is only suited for small values of $k$. 

When one is faced with the task of working with high values of $k$ (say, $k > 5$), and a small number of entries for each mode -- a natural choice
are instead tensor trains \cite{oseledets2011tensor}
or the hierarchical Tucker decomposition \cite{kressner2014htucker}. 

We shall concentrate on the former choice, and in the next section
we briefly recall the main tools that we use in the rest of the paper. 

\subsection{Tensor-train format}

Tensor trains are a technology aimed at treating
high-dimensional problems: they have  already been successfully
applied to Markov chains (see \cite{bolten2016multigrid,kressner2014low})
and to the numerical solution of high-dimensional PDEs (\cite{dolgov2012fast,kazeev2012lowrank} and  \cite{lubich2015time}). 

Let us consider a large system composed
by $k$ smaller components, each with $n_i$ states, $i = 1, \ldots, k$. The
potential state space $\mathcal {PS}$ can then be written as 
\[
  \mathcal {PS} := \{ 1, \ldots, n_1 \} \times \ldots \times
    \{ 1, \ldots, n_k \}. 
\]
At each time $t$, the probability vector $\pi(t)^T = \pi_0^T e^{tQ}$ can
be expressed in tensor form, as an array with $k$ indices
$\pi(t) = \pi(i_1, \ldots, i_k)$. A tensor train representation of 
a tensor $v$
is a collection of order $3$ tensors $M_i$ of size 
$r_{i} \times n_i \times r_{i+1}$ such that\footnote{In particular, 
	$M_1$ and $M_k$ are matrices, instead of order
	$3$, because they have one dimension with only $1$ index.} $r_1 = r_k = 1$, and 
\begin{equation} \label{eq:tt}
  v_{i_1,\ldots,i_k} = \sum_{t_1, \ldots, t_{k-1}} 
    M_1 (1, i_1, t_1) M_2(t_1, i_2, t_2) \ldots  M_k(t_{k-1}, i_k, 1)
\end{equation}
The tensors $M_{i}$ are called \emph{carriages}, hence the name
\emph{tensor train} \cite{oseledets2011tensor}. 
The tuple $(r_2, \ldots, r_{k-1})$ is called the
TT-rank of the tensor $v$. Similarly, matrices $m \times n$ 
can be represented as tensors by subdividing the row and column indices. 
More precisely, a matrix $A \in \mathbb C^{n_1 \ldots n_k \times n_1 \ldots n_k}$
can be viewed (up to permuting the indices) as a larger vector 
of the vector space $\mathbb C^{m_1n_1 \times \ldots \times m_k n_k}$. 
This vector can be stored in the TT-format as described in \eqref{eq:tt}. Such
arrangement makes computing matrix-vector product and matrix-matrix product
relatively simple to implement; we refer the reader to \cite{oseledets2011tensor} for
further details. 

On the software side, a well-established framework \cite{tttoolbox} is available for Python
and MATLAB \cite{MATLAB18} 
We rely on the latter  for our numerical experiments. 

\subsection{Exponential sums}

Given a Kronecker sum $\mathcal A := A_1 \oplus \cdots \oplus A_k$, in a few
cases of interest one can devise an efficient strategy for evaluating 
its inverse (or the action of the inverse on some vector). 
Given the relevance of this problem
in high-dimensional PDEs and various other settings
of applied mathematics, several approaches have been devised over
the years. In this section we briefly recall the one known
after the name of \emph{exponential sums} \cite{braess2005approximation}.

For the sake of self-completeness, 
we briefly recall the main important facts related
to this topic. The idea behind exponential sums is to rephrase the
inverse as a combination of matrix exponentials. The latter are much easier
to compute for a Kronecker sum, as stated by the next Lemma. 

\begin{lemma}
  Let $\mathcal A = A_1 \oplus \ldots \oplus A_k$ be a Kronecker sum. Then,
  the matrix exponential $e^{\mathcal A}$ is given by
  \[
    e^{\mathcal A} = e^{A_1} \otimes e^{A_2} \otimes \cdots \otimes e^{A_k}. 
  \]
\end{lemma}

\begin{proof}
  It suffices to recall that $e^{A + B} = e^A e^B$ whenever $A$ and $B$
  commute; clearly, all the addends in the sum defining $\mathcal A$
  commute, and the result follows by $e^{A \otimes B} = e^{A} \otimes e^{B}$. 
\end{proof}

In view of the previous result, assume we know an expansion of $\frac{1}{x}$
of the following form:
\[
  \frac{1}{x} = \sum_{j = 1}^{\infty} \alpha_j e^{-\beta_j x}, \qquad
  \forall x \in \Lambda(\mathcal A),
\]
where $\Lambda(\cdot)$ denote the spectrum of the operator. Then,
\[
  \mathcal A^{-1} = \sum_{j = 1}^{\infty} \alpha_j e^{-\beta_j \mathcal A}.
\]
Truncating the above series yields an approximation of the inverse, and the
matrix exponentials are very cheap to compute if one knows the factors $A_i$, 
since these are relatively small matrices that can be handled with
dense linear algebra techniques.  It
remains to construct a method to efficiently compute $\alpha_j$ and $\beta_j$ of
such an expansion. We say that
an exponential sum has accuracy $\epsilon > 0$ on the interval $[a, b]$ if, 
for any $x \in [a, b]$, we have $\left| \frac{1}{x} - \sum_{j = 1}^k \alpha_j e^{-\beta_j x}\right|\leq \epsilon$.

\begin{lemma}[\cite{highammatfun}]
	\label{lem:higham}
	Let $A$ be a diagonalizable matrix, $f(z)$ a function, 
	Then, 
	\[
	  \norm{f(A)} \leq C \max_{z \in \Lambda(A)} |f(z)|, 
	\]
	where $\norm{\cdot}$ is any induced norm. 
\end{lemma}

\begin{lemma} \label{lem:error}
	Let $\alpha_j, \beta_j$ be the coefficients of an exponential sum
	with accuracy $\epsilon$ over $[1, R]$. Then, if $A$ is 
	a diagonalizable matrix
	with spectrum contained in $[1, R]$, we have
	\[
	  \norm*{A^{-1} - \sum_{j = 1}^k \alpha_j e^{-\beta_j A}} \leq
	    C \epsilon 
	\], 
	where $C = \norm{V} \norm{V^{-1}}$, where $V$ is the matrix
	of eigenvectors of $A$, and $\norm{\cdot}$ is any induced norm. In particular, 
	if $\norm{\cdot}$ is the euclidean norm and $A$ is normal, then $C = 1$. 
\end{lemma}

\begin{proof}
	The result is a straightforward application of Lemma~\ref{lem:higham}
	to the function $f(z) = \frac 1z - \sum_{j = 1}^k \alpha_j e^{-\beta_j z}$ over the domain $[1, R]$. 
\end{proof}

A similar result for the non-diagonalizable case can be obtained relying on the
field of values, for which we refer to \cite{crouzeix2017numerical}.

Lemma~\ref{lem:error} implies that, given a (normal) matrix in Kronecker sum form
$\mathcal A = A_1 \oplus \ldots \oplus A_k$, to achieve an accuracy $\epsilon$ 
in the computation of $\mathcal A^{-1}$ or, equivalently, 
in the solution of the linear system $\mathcal A x = b$, we shall 
obtain an exponential sum with coefficients $\alpha_j, \beta_j$ achieving
that accuracy $\epsilon$ over the eigenvalues of $\mathcal A$. 
More precisely, we can state the following result that related the norm
a matrix function to the spectral properties of the matrices $A_i$. 

\begin{lemma}
	Let $\mathcal A = A_1 \oplus \ldots \oplus A_k$, where $A_i$ are diagonalizable for 
	$i = 1, \ldots, k$. Then, given a function $f(z)$ defined on $\Lambda(A_1) + 
	\ldots + \Lambda(A_k)$, we have 
	\[
	  \norm{f(\mathcal A)}_2 \leq \left( \prod_{1 \leq i \leq k} \norm{V_i}_2 \norm{V_i^{-1}}_2\right) \cdot \max \left\{ 
	    |f(\lambda_{i_1} + \ldots + \lambda_{i_k})| \ , \ 
	    \lambda_{i_j} \in \Lambda(A_j)
	  \right\},
	\]
	where $V_i$ is the matrix of the eigenvectors of $A_i$. 
\end{lemma}

\begin{proof}
	We note that $\mathcal A$ is diagonalizable if and only if $A_i$ are, and in
	this case the matrix of eigenvectors is given by $V_1 \otimes \ldots \otimes V_k$.
	Applying Lemma~\ref{lem:error} yields the bound 
	\[
	  \norm{f(\mathcal A)} \leq \norm{V} \cdot \norm{V^{-1}}_2 \cdot \max \{ 
	  |f(\lambda_{i_1} + \ldots + \lambda_{i_k})| \ , \ 
	  \lambda_{i_j} \in \Lambda(A_j). 
	\]
	The conclusion follows noting that $\norm{V}_2 = \prod_{1 \leq i \leq k} \norm{V_i}_2$, 
	and the similar result for its inverse. 
\end{proof}

The construction
of the coefficients $\alpha_j, \beta_j$ is beyond the scope of this
paper; our approach relies on \cite{braess2005approximation}, to which we refer
for further details.  In particular, we rely on the construction
of the exponential sums using sinc quadrature; since we know the exact 
condition number, we may use the precomputed tables available at \cite{hackbusch2019}, which cannot be easily computed on the fly, since they
require a specially adapted Newton method with extended precision to reach
high accuracy. This would potentially add another speed up to the code, 
since they have a faster convergence rate.

\section{Computing the MTTA}
\label{sec:MTTA}

In \cite{masetti2018computing} it has been shown that the computation of
several performability measures can be recast as the evaluation of
a matrix function. In most cases, one has to compute $w^T f(Q) v$
for appropriate vectors $v, w$, and a certain function $f(z)$. In this
work, we focus on the computation of the \emph{mean-time-to-absorption}.

We assume that the states of the Markov chain are labeled with
the integers from $1$ to $N = |\mathcal{PS}|$, and that there
is a single absorbing state, which can be assumed to have index $N$. 
Then, we compute the quantity
\begin{equation}
\label{eq:defMTTA}
  \mtta = \int_0^{\infty} \mathbb P \{ X(\tau) < N \}\ d\tau =
  \mathbb E \left[ \int_0^\infty \mathbbm 1_{\{1,\dots,N-1\}}(X(\tau)) \ d\tau  \right]. 
\end{equation}
Often, we are interested in the case where the absorbing state corresponds
to the failure state of the system. In this case, we use the name
\emph{mean-time-to-failure} and the notation $\mathrm{MTTF}$. We
assume that the matrix $Q$ is partitioned as follows:
\[
  Q = \left[\begin{array}{ccc|c}
  &&& v_1\\
  & \hat Q & & \vdots \\
  &&& v_{N-1}\\  \hline
  0 & \dots & 0& 0 \\
  \end{array} \right], 
\]
where the last row is forced to be zero because the state $N$ is
absorbing. Following \cite{trivedi2017reliability}, we know that
\begin{equation}
  \label{eq:mtta}
  \mtta = - \hat \pi_0^T \hat Q^{-1} e =
  \pi_0^T f(Q) (e - e_Ne_N^T), \qquad
  f(z) =
  \begin{cases}
    -\frac{1}{z} & z \neq 0\\
    0           & \text{otherwise}
  \end{cases}
\end{equation}
where $e$ is the vector of all ones, and $\hat \pi_0$ contain
the first $N - 1$ entries of $\pi_0$. The aim is to consider
the case where $Q$ can be efficiently represented in the TT format.

\begin{remark} \label{rem:submatrix}
	The fact that $Q$ has a low-rank tensorial structure (in the TT sense)
	does not imply any particular structure for $\hat Q$. Indeed, the
	tensor structure requires making use of the isomorphism
	$\mathbb C^N = \mathbb C^{n_1} \times \ldots \times \mathbb C^{n_k}$, 
	which in turn is related to factorizing $N = n_1 \cdots n_k$. 
	Knowing the factors $n_i$ does not give any
	information on a similar factorization for $N - 1$, nor 
	on the low-rank carriages that might be used to represent $\hat Q$.
\end{remark}

In particular, a well-known method for transforming the MTTA problem
into the computation of a steady state vector of a irreducible
Markov chain, is to add a transition from the absorbing state
back to the starting one \cite{TB17}. However, this transforms the 
problem into an eigenvector computation of $Q + \delta Q$, where
$\delta Q$ has rank $1$. This problem might be addressed directly, 
exploiting the techniques presented in \cite{kressner2014low} 
for computing steady state probabilities in the TT-format; recasting
it into solving an augmented linear system (adding a row of ones to ensure
that the computed eigenvector is a probability) is instead undesirable, because
it would lead to a loss of tensor structure as reported in Remark~\ref{rem:submatrix}. 

Instead, here an alternative strategy is presented: 
we introduce an auxiliary matrix $S$ which allows to rephrase
the measure using the inverse of a low-rank perturbation of $Q$. 

\begin{lemma} \label{lem:S}
	Let $Q$ the infinitesimal generator of a continuous time Markov chain
	with exponential rates with $N$ states; assume that the state $N$ 
	is the only failure state, and let $S$ be the rank $1$
        matrix defined as 
	\[
	  S = (Q e_N) e_N^T - e_N e_N^T. 
	\]
	Then, if $\pi_0$ has the $N$-th component equal to $0$, we have 
	$\mtta = - \pi_0^T (Q - S)^{-1} e$, where $e$ is the vector of 
	all ones. 
\end{lemma}

\begin{proof}
  By construction, we have that $Q - S$ is block diagonal
  and therefore
  \[
    (Q - S)^{-1} =
    \begin{bmatrix}
      \hat Q \\
      & 1 \\
    \end{bmatrix}^{-1} =
    \begin{bmatrix}
      \hat Q^{-1} \\
      & 1 \\
    \end{bmatrix}, 
  \]
  and
  $-\pi_0^T (Q - S)^{-1} e = -\hat \pi_0^T \hat Q^{-1} e - [\pi_0]_N =
  \mtta - [\pi_0]_N$. We conclude noting that $ [\pi_0]_N = 0$. 
\end{proof}

\begin{remark}
Note that the TT-rank of $S$ is $(1, \ldots, 1)$, since it 
is a matrix of rank $1$ \cite{oseledets2011tensor}, 
and if $Q$ has a low TT-rank the same holds
for $Q - S$.   
\end{remark}

The important consequence of Lemma~\ref{lem:S} is that, even though we 
cannot extract a submatrix from $Q$ to compute the $\mtta$, we can make a 
rank $1$ (both in CP and in the TT sense) perturbation $S$ to $Q$, such
that $Q - S$ is invertible, and allows to easily obtain
the same result.

\begin{lemma} \label{lem:gamma}
  Let $v = e + \gamma e_N$, for any $\gamma \in \mathbb R$. Then, with
  the notation of Lemma~\ref{lem:S}, we have
  \[
    \mtta = -\pi_0^T (Q - S)^{-1} e = -\pi_0^T (Q - S)^{-1} v. 
  \]
\end{lemma}

\begin{proof}
  Note that, $[\pi_0]_N = 0$, and therefore
  \[
    -\pi_0^T (Q - S)^{-1} v =
    - \begin{bmatrix}
      & \hat\pi_0^T \hat Q^{-1} & & 0 
    \end{bmatrix}
    \begin{bmatrix}
      1 \\ \vdots \\ 1 \\ 1 + \gamma
    \end{bmatrix} = -\hat\pi_0^T \hat Q^{-1} e = \mtta.
  \]
\end{proof}

We have recast the problem to solving a linear system 
$(Q - S) x = e$, where the matrix $Q - S$ is expressed in TT format. 
Unfortunately, as we will see later on, a few problem of interest
for us do not play very well together with the more widespread 
tensor train system solvers (such as AMEN \cite{dolgov2014alternating} or DMRG \cite{oseledets2012solution}). For this
reason, we propose a different solution scheme based on the
\emph{Neumann expansion}. In particular, let $M$ be any matrix
with spectral radius strictly smaller than $1$. Then, 
\begin{equation} \label{eq:neumann}
  (I - M)^{-1} = I + M + M^2 + M^3 + \ldots = \sum_{j = 0}^\infty M^j
\end{equation}
If we partition $Q$ as $Q = Q_1 + Q_2$, we can write
\[
  (Q - S)^{-1} = (Q_1 + Q_2 - S)^{-1} = (I + Q_1^{-1}(Q_2 - S))^{-1} Q_1^{-1}. 
\]
Setting $M = - Q_1^{-1}(Q_2 - S)$, assuming that $\rho(M) < 1$ and
using the Neumann expansion \eqref{eq:neumann} we obtain
\begin{align}  
  (Q - S)^{-1} = \sum_{j = 0}^\infty (-1)^j (Q_1^{-1} (Q_2 - S))^j Q_1^{-1}. 
\end{align}
The above formula can be used to approximate 
$x = (Q - S)^{-1} e$ as needed for \eqref{eq:mtta} by truncating the
infinite sum to $\ell$ terms:
\[
  x_\ell = \sum_{j = 0}^\ell (-1)^j (Q_1^{-1} (Q_2 - S))^j Q_1^{-1} e, \qquad 
  \norm{x - x_\ell}_\infty \leq C \rho(M)^{\ell+1}, 
\]
for an appropriate constant $C$. 
We give an explicit method to construct
the additive splitting $Q = Q_1 + Q_2$ so that $M$ is guaranteed
to have spectral radius less than $1$. This will be achieved in
Theorem~\ref{thm:neumann}. The pseudocode describing the resulting
method is presented in Algorithm~\ref{alg:neumann}.

\begin{algorithm}
  \caption{Neumann series \eqref{eq:neumann} to approximate $x = (Q - S)^{-1}e$} \label{alg:neumann}
  \begin{algorithmic}[1]
    \Procedure{NeumanSeries}{$Q_1,Q_2,\ell$}
    \State $y \gets Q_1^{-1} e$
    \State $x \gets y$
    \For{$j = 1, \ldots, \ell$}
    \State $y \gets - Q_1^{-1} (Q_2 - S) y$
    \State $x \gets x + y$
    \EndFor
    \State \Return $x$  
    \EndProcedure
  \end{algorithmic}
\end{algorithm}

This method has a linear convergence rate \cite{demmel1997applied},
which is given by $\rho(M)$. However, it can be accelerated
to obtain a quadratically convergent method by a simple modification. Note
that we can refactor (\ref{eq:neumann}) as follows:
\begin{equation} \label{eq:neumann2}
  (I - M)^{-1} = (I + M) (I + M^2) (I + M^4) \cdots (I + M^{2^j}) \cdots
\end{equation}
Truncating the above equation and permuting the factors $(I + M^{2^j})$
yields another method to approximate
$x = (Q - S)^{-1}e$, which has a much faster convergence, and is
described by the equation:
\[
  (I - M)^{-1} Q_1^{-1} e \approx x_{2^{\ell+1}-1} = (I + M^{2^\ell}) (I + M^{2^{\ell-1}}) \cdots (I + M^2) (I + M) Q_1^{-1} e
\]
In particular, the result of $\ell$ steps of this method gives
the same result of $2^{\ell+1}-1$ of Algorithm~\ref{alg:neumann}. 
The pseudocode for this approach is given in Algorithm~\ref{alg:neumann2}. As
visible on line~\ref{alg:M2}, this method required to store the repeated squares
of the matrix $M$. 

\begin{algorithm}
  \caption{Neumann series \eqref{eq:neumann2} to approximate $x = (Q - S)^{-1}e$}
  \label{alg:neumann2}
  \begin{algorithmic}[1]
    \Procedure{NeumanSeries}{$Q_1,Q_2,\ell$}
    \State $M \gets -Q_1^{-1}(Q_2 - S)$
    \State $x \gets Q_1^{-1} e + M Q_1^{-1} e$
    \For{$j = 2, \ldots, \ell$}
    \State $M \gets M^2$ \label{alg:M2}
    \State $x \gets x + Mx$
    \EndFor
    \State \Return $x$  
    \EndProcedure
  \end{algorithmic}
\end{algorithm}

\begin{remark}
  A favorable property of both approaches is that the convergence of
  $x_\ell$ to $x$ is monotonically decreasing and non-positive. That is,
  for each $\ell' \leq \ell$ we have $x_{\ell'} \geq x_\ell$. Since the $\mtta$\
  is equal to $-\pi_0^T x$, 
  the estimates $-\pi_0^T x_\ell$ of the $\mtta$\ obtained in the intermediate steps are
  guaranteed lower bounds.
\end{remark}

We note that Algorithm~\ref{alg:neumann} can be slightly modified
to compute $\pi_0^T(Q - S)^{-1}$ instead of $(Q - S)^{-1} e$. Both
vectors can then be used to compute the MTTA through a dot product. However, the former choice has the advantage that
$\pi_0(s) = 0$  implies that $x(s) = 0$ throughout the iterations for
every $s \in \mathcal{PS}\setminus \mathcal{RS}$. In particular, 
the non-reachable part of the chain has no effect on the computed
tensor, and this helps to keep the TT-rank low during the iterations. 
The modified algorithm is reported for completeness in Algorithm~\ref{alg:neumannt}, where now $x,y$ are row vectors. 

\begin{algorithm}
	\caption{Neumann series \eqref{eq:neumann} to approximate $\pi^T = \pi_0^T (Q - S)^{-1}$} \label{alg:neumannt}
	\begin{algorithmic}[1]
		\Procedure{NeumanSeries}{$Q_1, Q_2,\ell$}
		\State $y \gets \pi_0^T$
		\State $x \gets y$
		\For{$j = 1, \ldots, \ell$}
		\State $y \gets - y Q_1^{-1} (Q_2 - S)$
		\State $x \gets x + y$
		\EndFor
		\State $x \gets x Q_1^{-1}$
		\State \Return $x$  
		\EndProcedure
	\end{algorithmic}
\end{algorithm}

\begin{lemma} \label{lem:non-negative}
	Let $A \geq 0$ be a non-negative $N \times N$ 
	matrix, and $e$ the vector of all ones. Then, 
	\[
	  \rho(A) \leq \norm{A}_\infty = \max_{1 \leq i \leq N} (A e)
	\]
	Moreover, if there is at least one component of $A e$ strictly
	smaller than $\norm{A}_\infty$, then $\rho(A) < \norm{A}_\infty$. 
\end{lemma}

\begin{proof}
  The first statement is the definition of infinity norm, whereas the
  second follows directly by the first Gerschgorin theorem.
\end{proof}

The next result provides a technique for splitting an infinitesimal
general $Q$ (i.e., up to a change of sign,
any $M$-matrix with zero row sums) in a way that allow to perform
the Neumann expansion. Let us first recall a few properties
of diagonally dominant matrices. 

This observation implies that, even though in view of Remark~\ref{rem:submatrix} we have not restricted the problem
to the set of reachable states, 
when running Algorithm~\ref{alg:neumannt} the iteration is implicitly
restricted to this set, 
as $x_\ell$ are vectors with positive components only for indices
in $\mathcal{RS}$. 

The next results are aimed at constructing the additive splitting
for $Q$ that satisfies the hypotheses for the Neumann expansion. 

\begin{lemma} \label{lem:diagdomsign}
	Let $A = D - B$, with $D < 0$ and diagonal, $B \geq 0$, and 
	$Be < -D e$, where $e$ is the vector with all components equal to $1$. Then, $A$ is invertible and $A^{-1} \leq 0$. 
\end{lemma}

\begin{proof}
    This fact can be easily prove using the tools from theory 
    of nonnegative matrices, since $A$ is an $M$-matrix
    and $D - B$ is a regular splitting (see for instance \cite{plemmons1994nonnegative}). We provide a simple 
    proof for the sake of completeness. Since $D < 0$, the 
    condition $A^{-1} \leq 0$ is equivalent to $(D^{-1} A)^{-1} \geq 0$.
    Moreover, the strict row diagonal dominance implies that 
    $\norm{D^{-1} A}_{\infty} < 1$, so we have 
    \[
      (D^{-1} A)^{-1} = (I - (-D^{-1}A))^{-1} = \sum_{j \geq 0} 
        (-D^{-1}A)^j \geq 0, 
    \]
    where we have used that $(-D^{-1}A) \geq 0$. 
\end{proof}

\begin{theorem} \label{thm:neumann}
	Let $A = D + A_1 + A_2$ any $N \times N$ matrix such as $D$ is diagonal and non-positive, 
	$A_1, A_2$ are non negative, $e_N^T (D + A_1 + A_2) = e_N^T A_1 = 0$, and $(D + A_1 + A_2) e = 0$. Then, 
	if 
	we define $S = (A_1 + A_2) e_N e_N^T$, $(D + A_1)$ is invertible and $\min_{i = 1, \ldots, N-1} A_{iN} > 0$, 
	we have that $\norm{(D + A_1)^{-1} (A_2 - S)}_\infty < 1$. 
\end{theorem}

\begin{proof}
	Let us denote with $M := (D + A_1)^{-1} (A_2 - S)$. All the columns of this matrix are
	non-positive (see Lemma~\ref{lem:diagdomsign}), with the only exception of the last one, which is non-negative. This is a
	consequence of the fact that $(D + A_1)^{-1}$ is non-positive, and $(A_2 - S)$ has 
	the first $N - 1$ columns with positive entries, and the last one with negative ones. 
	
	Therefore, it is clear that we have $\norm{M}_\infty = \norm{M(e - 2e_N)}_\infty$, 
	and the vector $M(e - 2e_N)$ is element-wise non-positive by construction. 
	We have
	\[
	  M(e - 2e_N) = (D + A_1)^{-1} (A_2 - S) (e - 2e_N). 
	\]
	Using the relations $(D + A_1)^{-1} A_2 e = (D + A_1)^{-1} (D + A_1 + A_2) e - e = -e$ and 
	$Se = Se_N = (A_1 + A_2)e_N$ we get 
	\begin{align*}
	  M(e - 2e_N) &= -e -2 (D + A_1)^{-1} A_2 e_N + (D + A_1)^{-1} (A_1 + A_2) e_N\\
	             &= -e + (D + A_1)^{-1} (A_1 - A_2) e_N \\
	             &=  -e + e_N - (D + A_1)^{-1} (D + A_2) e_N. 
	\end{align*}
	By construction, we know that the last row of $D + A_2$ is equal to $-e_N^T A_1$ 
	and is therefore zero. The first $N - 1$ entries in $(D + A_2) e_N$ are taken from
	$A_2$ and therefore they are (strictly) positive. Since $(D + A_1)^{-1}$ is 
	entry-wise strictly negative, we have that $v = -(D + A_1)^{-1} (D + A_2) e_N$
	has the first $N - 1$ components strictly positive. Therefore, we have 
	that $M(e - 2e_N) > -1$ element-wise, and on the other hand we knew that 
	$M(e - 2e_N)$ is non-positive. This implies that $\norm{M}_\infty < 1$, 
	as claimed. 
      \end{proof}
  
  \begin{remark}
  	The previous results are closely related with the theory of $M$-matrices. Indeed, 
  	the decomposition $A = M - N$ with $M = (D + A_1)$ and $N = -A_2 + S$ is almost
  	a regular splitting, because $M^{-1}$ is negative, and $N$ is positive, with
  	the only exception of the last column. If it were a regular splitting, 
  	then this would automatically imply that $\rho(M^{-1}N) < 1$ -- in view
  	of the theory of nonsingular $M$-matrices \cite{plemmons1994nonnegative}. 
  \end{remark}

      \begin{lemma} \label{lem:positive}
        Let $y$ be any vector. Then, using the notation
        of Theorem~\ref{thm:neumann},
        $(D + A_1)^{-1}(A_2 - S)y$ has the
        last component equal to zero. Moreover, let $y_0$
        be any vector, and define
        \[
          y_{\ell+1} = -(D + A_1)^{-1}(A_2 - S) y_\ell, \qquad \ell \geq 1. 
        \]
        Then, if $y_\ell$ for $\ell > 0$ is non-negative we have
        $y_{\ell'} \geq 0$ for any $\ell' \geq \ell$. 
      \end{lemma}

      \begin{proof}
      	We start showing that $e_N^T(D + A_1)^{-1} (A_2 - S) = 0$, 
      	which proves the first claim. We have 
      	\begin{align*}
      		e_N^T(D + A_1)^{-1} (A_2 - S) &= D_{NN}^{-1} e_N^T (A_2 - S)
      		= e_N^T A_2 - e_N^T A_1 e_N e_N^T - e_N^T A_2 e_N e_N^T \\
      		&= e_N^T A_2 - e_N^T A_2 e_N e_N^T
      		= e_N^T A_2 (I - e_N e_N^T) \\
      		&= - e_N^T D (I - e_N e_N^T) = 0, 
      	\end{align*}
      	where we have used the properties $e_N^T (D + A_1 + A_2) = e_N^T + A_1 = 0$, and the definition of $S = (A_1 + A_2) e_N e_N^T$. 
      	
      	Assume now that $e_N^T y = 0$. 
      	Then, $z := (A_2 - S) y = A_2 y \geq 0$, since $Sy = 0$. Moreover, $(D + A_1)^{-1}$ is non-positive in view of Lemma~\ref{lem:diagdomsign}, 
      	and therefore $-(D + A_1)^{-1} z \geq 0$, concluding the proof. 
      \end{proof}

      \begin{remark}
        Note that choosing $\gamma = -1$ in the notation of \Cref{lem:gamma} provides
        a starting vector for the Neumann iteration \Cref{eq:neumann} that satisfies
        the hypotheses of \Cref{lem:positive}. Therefore, in this case the iteration
        to approximate the $\mtta$\ is monotonically increasing, and at the step $\ell$
        gives a lower bound for the final value of the $\mtta$.
      \end{remark}

\begin{theorem} \label{thm:constructivesplitting}
	Let $Q = \Delta + R + W$ be an infinitesimal generator of a Markov 
	chain as described in \eqref{eq:descriptorQ}, with 
	\[
	R = R^{(1)} \oplus \ldots \oplus R^{(k)}, \qquad 
	\Delta = -\mathrm{diag}( (W + R) e ),
      \] 
      and $\gamma \geq \norm{\Delta}_\infty$. Then, if we define
      $D := - \gamma I, A_1 = R$, 
      and $A_2 = W + (\Delta - \gamma I)$,
      these matrices satisfy the hypotheses of Theorem~\ref{thm:neumann}
      and there exists $\alpha_j, \beta_j$ such that
          \[
            D + A_1 = \left(R^{(1)} - \frac{\gamma}{k} I\right) \oplus 
            \left(R^{(2)} - \frac{\gamma}{k} I\right)
              \oplus \cdots \oplus \left(R^{(k)} - \frac{\gamma}{k} I\right),
          \]
     and therefore
     \[
       (D + A_1)^{-1} (A_2 - S) \approx 
       \sum_{j = 1}^\ell \alpha_j \left(
         e^{\beta_j (R_1  - \frac{\gamma}{k} I)} \otimes \ldots \otimes e^{\beta_j (R_k  - \frac{\gamma}{k} I)} 
       \right) (A_2 - S). 
     \]
   \end{theorem}

\begin{remark}
	We note that the choice of $\gamma$ allows to control the condition
	number of the matrix $D + A_1$; our experience shows that larger
	values for $\gamma$ (which give lower condition numbers), 
	provide slower convergence with $\rho$ approaching $1$, 
	but also lower TT-ranks during the 
	Neumann iteration. This choice is discussed in further
	detail in Section~\ref{sec:gamma}. 
\end{remark}

\section{Computational remarks}
\label{sec:remarks}

In this section we report a few computational remarks concerning
our implementation. The variants of the Neumann
expansion described in Algorithm~\ref{alg:neumann2} and \ref{alg:neumannt}
have been implemented in the toolbox \texttt{kaes}, which is
freely available\footnote{\url{https://github.com/numpi/kaes/}.}. The
toolbox is implemented in MATLAB, and given cell-arrays \texttt{R}, 
\texttt{W} containing
the factors defining $Q$, one may compute the value of the MTTA by calling 
\texttt{m~=~eval\_measure('inv', pi0, r, R, W)} where \texttt{pi0}
and \texttt{r} contain the initial probability distribution and
the reward vector. The function has some optional parameters, 
that allows to tune the required tolerance and the value of 
$\gamma$. 

A few considerations can be helpful in trying to obtain maximum 
performances from the implementation. 

\subsection{Ordering of the subsystems}

Since the TT representation represents the interaction between between
the subsystem $i$ and $i+1$ in each carriage, we have found that it
is beneficial to reorder the topology so that few nodes are linked 
to far ones. 

In particular, given the adjacency matrix $\mathcal T$ that 
represents the connection graph (i.e., $\mathcal T_{ij}) = 1$ if and only
if there is an edge in graph from the node $i$ to the node $j$), it 
can be helpful to reorder the subsystems to make this matrix
as banded as possible. To this end, we have employed the reverse
Cuthill-McKee ordering implemented in MATLAB in the function
\texttt{symrcm}.

\subsection{The choice of $\gamma$} \label{sec:gamma}

The choice of the parameter $\gamma$ in Theorem~\ref{thm:constructivesplitting} can have important effects
on the performance of the algorithm. 

We have verified that choosing
$\gamma$ relatively large, for instance $\gamma \gg \norm{\Delta}_\infty$, 
can be helpful. This reduces the conditioning of the matrix to invert
to a small number, and thus very few exponential sums are needed 
to achieve a very high accuracy.  More importantly, this helps
to keep the TT-ranks low during the iteration, especially when applying 
\eqref{eq:neumann2}, which in turn suffers very mildly from having
the spectral radius close to $1$ (thanks to the
quadratic convergence rate). On the other hand, when applying 
the linearly convergence iteration \eqref{eq:neumann}, the
minimal choice $\gamma = \norm{\Delta}_\infty$ is often advisable. 
Indeed, for this iteration a $\rho$ close to $1$ is much more harmful, 
and in general it can be quite memory efficient (it only work
with compressed vectors). 

We do not have a ``universal recipe'' for these choices, so 
it might be helpful to do some preliminary parameter tuning on small
problems of a given class before tackling the large scale cases. We
plan to further investigate this matter in the future. 

\section{Case study}
\label{sec:caseStudy}

Consider a cyber-physical system comprising $k$ components,
consisting each of a mechanical object and a Monitoring and Control Unit (MCU).
The mechanical objects are independent one from the other 
whereas the working status of the MCU software on component $j$ 
depends on data produced by local sensors
and can depend also on data coming from the MCU of component $i$.  
Thus, it is possible to define a topology of interactions among
component MCUs: define $\mathcal T$ the $k\times k$ matrix as
$\mathcal T(i,j)=1$ if $i=j$ or the $j$-th MCU consumes data produced by the
$i$-th MCU.
At every time instant, the MCU and the mechanical object
on each component can be \emph{working} or \emph{failed}. 
If the mechanical object on component $i$ fails then instantaneously
also the MCU on component $i$ fails.
If $\mathcal T(i,j)=1$ then the failure of the $i$-th MCU implies
an instantaneous failure of the $j$-th MCU.
The failure time of the software running on the $i$-th MCU is assumed
to be exponentially distributed with rate $\lambda^s_i$. 

The MCU on component $i$ can modify the behaviour of the mechanical object
on component $i$, so the failure time of the mechanical object 
on component $i$ is exponentially distributed with rate $^\bullet\lambda^h_i$
if the MCU on component $i$ is working, and $^\circ\lambda^h_i$ if the MCU
is already failed.
No repair is considered.

We are interested in evaluating the Mean Time to System Failure, where
the system is considered failed when all the mechanical objects are failed.
\begin{figure}
	\centering
	\begin{tikzpicture}[node distance=6em,>=latex',bend angle=45,
	pre/.style={<-,shorten <=0pt,>=latex'},
	post/.style={->,shorten >=0pt,>=latex'}]
	
	\node[draw=none] (origin) at (0,0) {};
	
	\node[circle, draw=black, minimum size=3em] (m1s1) 
	at ($(origin)+(0em,0em)$) {};
	\node[draw=none] (m1s1s) at ($(m1s1)+(0,.5em)$) {s$_1$ w};
	\node[draw=none] (m1s1h) at ($(m1s1)+(0,-.5em)$) {h$_1$ w};
	\node[circle, draw=black, minimum size=3em] (m1s2) 
	at ($(m1s1)+(-3em,-4em)$) {};
	\node[draw=none] (m1s2s) at ($(m1s2)+(0,.5em)$) {s$_1$ f};
	\node[draw=none] (m1s2h) at ($(m1s2)+(0,-.5em)$) {h$_1$ w};
	\node[circle, draw=black, minimum size=3em] (m1s3) 
	at ($(m1s1)+(3em,-4em)$) {};
	\node[draw=none] (m1s3s) at ($(m1s3)+(0,.5em)$) {s$_1$ f};
	\node[draw=none] (m1s3h) at ($(m1s3)+(0,-.5em)$) {h$_1$ f};
	\path (m1s1) edge [->] node [right] {$\lambda^s_1$} (m1s2);
	\path (m1s2) edge [->] node [below] {$^\circ\lambda^h_1$} (m1s3);
	\path (m1s1) edge [->] node [right] {$^\bullet\lambda^h_1$} (m1s3);
	\draw ($(m1s2)+(-2em,-3em)$) rectangle ($(m1s1)+(5em,3em)$);
	\node[draw=none] at ($(m1s1)+(-2em,3.5em)$) {Component 1};
	
	\node[circle, draw=black, minimum size=3em] (m2s1) 
	at ($(origin)+(12em,0em)$) {};
	\node[draw=none] (m2s1s) at ($(m2s1)+(0,.5em)$) {s$_2$ w};
	\node[draw=none] (m2s1h) at ($(m2s1)+(0,-.5em)$) {h$_2$ w};
	\node[circle, draw=black, minimum size=3em] (m2s2) 
	at ($(m2s1)+(-3em,-4em)$) {};
	\node[draw=none] (m2s2s) at ($(m2s2)+(0,.5em)$) {s$_2$ f};
	\node[draw=none] (m2s2h) at ($(m2s2)+(0,-.5em)$) {h$_2$ w};
	\node[circle, draw=black, minimum size=3em] (m2s3) 
	at ($(m2s1)+(3em,-4em)$) {};
	\node[draw=none] (m2s3s) at ($(m2s3)+(0,.5em)$) {s$_2$ f};
	\node[draw=none] (m2s3h) at ($(m2s3)+(0,-.5em)$) {h$_2$ f};
	\path (m2s1) edge [->] node [right] {$\lambda^s_2$} (m2s2);
	\path (m2s2) edge [->] node [below] {$^\circ\lambda^h_2$} (m2s3);
	\path (m2s1) edge [->] node [right] {$^\bullet\lambda^h_2$} (m2s3);
	\path (m2s1) edge [bend right,->, dotted] node [left] {$\lambda^s_1$} (m2s2);
	\draw ($(m2s2)+(-2em,-3em)$) rectangle ($(m2s1)+(5em,3em)$);
	\node[draw=none] at ($(m2s1)+(-2em,3.5em)$) {Component 2};
	
	\node[circle, draw=black, minimum size=3em] (m3s1) 
	at ($(origin)+(0em,-12em)$) {};
	\node[draw=none] (m3s1s) at ($(m3s1)+(0,.5em)$) {s$_2$ w};
	\node[draw=none] (m3s1h) at ($(m3s1)+(0,-.5em)$) {h$_2$ w};
	\node[circle, draw=black, minimum size=3em] (m3s2) 
	at ($(m3s1)+(-3em,-4em)$) {};
	\node[draw=none] (m3s2s) at ($(m3s2)+(0,.5em)$) {s$_2$ f};
	\node[draw=none] (m3s2h) at ($(m3s2)+(0,-.5em)$) {h$_2$ w};
	\node[circle, draw=black, minimum size=3em] (m3s3) 
	at ($(m3s1)+(3em,-4em)$) {};
	\node[draw=none] (m3s3s) at ($(m3s3)+(0,.5em)$) {s$_2$ f};
	\node[draw=none] (m3s3h) at ($(m3s3)+(0,-.5em)$) {h$_2$ f};
	\path (m3s1) edge [->] node [right] {$\lambda^s_3$} (m3s2);
	\path (m3s2) edge [->] node [below] {$^\circ\lambda^h_3$} (m3s3);
	\path (m3s1) edge [->] node [right] {$^\bullet\lambda^h_3$} (m3s3);
	\path (m3s1) edge [bend right,->, dotted] node [left] {$\lambda^s_1$} (m3s2);
	\draw ($(m3s2)+(-2em,-3em)$) rectangle ($(m3s1)+(5em,3em)$);
	\node[draw=none] at ($(m3s1)+(-2em,3.5em)$) {Component 3};
	
	\node[circle, draw=black, minimum size=3em] (m4s1) 
	at ($(origin)+(12em,-12em)$) {};
	\node[draw=none] (m4s1s) at ($(m4s1)+(0,.5em)$) {s$_2$ w};
	\node[draw=none] (m4s1h) at ($(m4s1)+(0,-.5em)$) {h$_2$ w};
	\node[circle, draw=black, minimum size=3em] (m4s2) 
	at ($(m4s1)+(-3em,-4em)$) {};
	\node[draw=none] (m4s2s) at ($(m4s2)+(0,.5em)$) {s$_2$ f};
	\node[draw=none] (m4s2h) at ($(m4s2)+(0,-.5em)$) {h$_2$ w};
	\node[circle, draw=black, minimum size=3em] (m4s3) 
	at ($(m4s1)+(3em,-4em)$) {};
	\node[draw=none] (m4s3s) at ($(m4s3)+(0,.5em)$) {s$_2$ f};
	\node[draw=none] (m4s3h) at ($(m4s3)+(0,-.5em)$) {h$_2$ f};
	\path (m4s1) edge [->] node [right] {$\lambda^s_4$} (m4s2);
	\path (m4s2) edge [->] node [below] {$^\circ\lambda^h_4$} (m4s3);
	\path (m4s1) edge [->] node [right] {$^\bullet\lambda^h_4$} (m4s3);
	\path (m4s1) edge [bend right=20,->, dotted] node [left] {$\lambda^s_3$} (m4s2);
	\path (m4s1) edge [bend right=65,->, dotted] node [left] {$\lambda^s_2$} (m4s2);
	\draw ($(m4s2)+(-2em,-3em)$) rectangle ($(m4s1)+(5em,3em)$);
	\node[draw=none] at ($(m4s1)+(-2em,3.5em)$) {Component 4};
	
	\node (topology) 
	at ($.5*(m2s3)+.5*(m4s1)+(11em,0)$) {
	  $\begin{bmatrix}
	    1& 1& 1& 0\\
	    0& 1& 0& 1\\
	    0& 0& 1& 1\\
	    0& 0& 0& 1\\
	  \end{bmatrix}$
	};
	\node[draw=none,align=center] at ($(topology)+(0,4em)$) {Topology of\\interdependencies:};
	\node[draw=none] at ($(topology)+(-4em,0em)$) {$\mathcal T=$};
	
	\end{tikzpicture}
	\caption{\label{fig:caseStudy}Example SAN for the case study
	where there are $4$ components. Dotted arrows are 
	synchronization transitions.}
\end{figure}
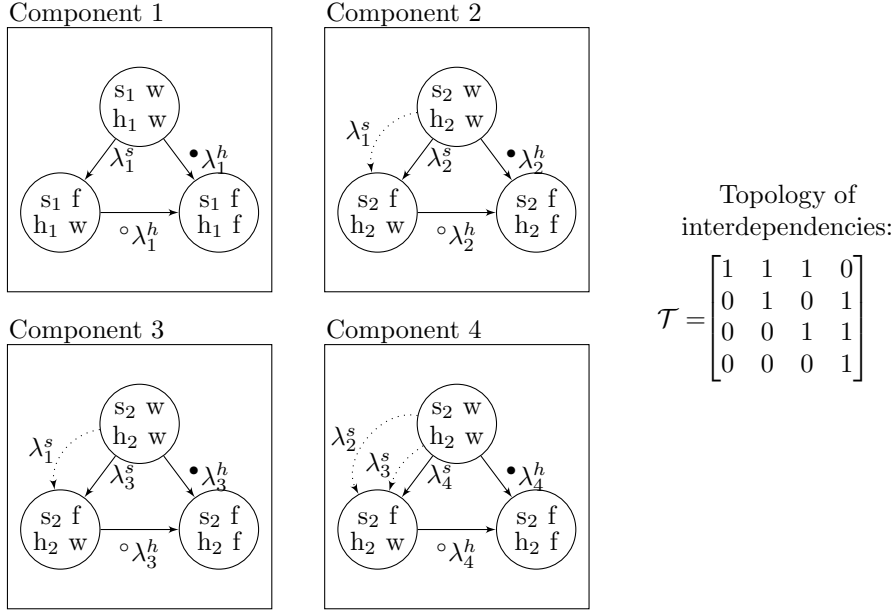
\Cref{fig:caseStudy} depicts the SAN model for a simple case where there are
$4$ components.
In particular, the state of component $i$, represented with a circle,
is defined by software status (s$_i$ is \texttt{w} if working,
\texttt{f} if failed) and the hardware status (h$_i$ is \texttt{w} if working,
\texttt{f} if failed). Transitions can be local or synchronized, represented as
arrows and dotted arrows, respectively, and labelled by their rate.  
Each component model has $3$ states, so that $|\mathcal{PS}|=3^k$,
and the cardinality of $\mathcal{RS}$ depends on $\mathcal T$. 
Notice that, if $\mathcal T=I$ then $\mathcal{RS}=\mathcal{PS}$, whereas 
the size of $\mathcal{RS}$ decreases as the number of interactions increases. 
The local and synchronization contribution matrices are then obtained
as in~\Cref{eq:RW}. 
\begin{align}
R^{(i)} &= 
\begin{bmatrix}
0& 0& ^\bullet\lambda^h_i\\
0& 0& ^\circ\lambda^h_i\\
0& 0& 0
\end{bmatrix}\label{eq:RW}, & 
W^{(t_j,i)} &= 
\begin{cases}
	\begin{bmatrix}
	0& \lambda^s_j& 0\\
	0& 0& 0\\
	0& 0& 0
	\end{bmatrix} &\text{if }i=j\\
	\begin{bmatrix}
	0& 1& 0\\
	0& 1& 0\\
	0& 0& 1
	\end{bmatrix} & {\text{if }i\neq j\text{ and } \mathcal T(j,i)=1}\\
	\begin{bmatrix}
	1& 0& 0\\
	0& 1& 0\\
	0& 0& 1
	\end{bmatrix} &\text{otherwise}
\end{cases}
\end{align}

\section{Experimental results}
\label{sec:results}

The case study model presented in~\Cref{sec:caseStudy}
has been implemented in MATLAB~\cite{MATLAB18} and studied 
applying the method discussed so far.
In particular, we consider the following set of parameters:
\[
^{\bullet}\lambda^h_i=\frac{i}{10},\quad
^{\circ}\lambda^h_i=i,\quad
\lambda^s_i=i,
\]
and the topology $\mathcal T$ has been chosen at random
with the following constraints:
\begin{itemize}
	\item each component has impact on itself, i.e., 
	$\mathcal T(i,i)=1$,
	\item For each $i \neq j$, the entry $\mathcal T(i,j)$ is set to $1$ with
	probability $\frac{1}{2k}$.
\end{itemize}

More precisely, the sparse
matrix $\mathcal T$ has been generated using the MATLAB command \texttt{T = (speye(k)+sprand(k,k,.5/k)) > 0}. 

\begin{table}[t]
	\centering 	\small
	\pgfplotstabletypeset[%
	column type=c,
	every head row/.style={
		before row={
			\toprule
			&\multicolumn{1}{c|}{}
			& \multicolumn{2}{c|}{Algorithm~\ref{alg:neumann2}} & 
			\multicolumn{2}{c|}{Algorithm~\ref{alg:neumannt}} & 
			\multicolumn{3}{c|}{AMEn} & 
			\multicolumn{3}{c}{DMRG} \\ 
		}, 
		after row = \midrule 
		},
	every last row/.style={after row=\bottomrule},		
	sci zerofill,
	columns={0,1,2,5,6,9,10,12,13,14,16},
	columns/0/.style={dec sep align={c|},column type/.add={}{|},column name=$k$,fixed},
	columns/1/.style={column name=Avg},
	columns/2/.style={column name=Max,column type/.add={}{|}},
	columns/5/.style={column name=Avg},
	columns/6/.style={column name=Max,column type/.add={}{|}},
	columns/9/.style={column name=Avg},
	columns/10/.style={column name=Max},
	columns/12/.style={column name=OOM,column type/.add={}{|},postproc cell content/.append code={
			\pgfkeysalso{@cell content/.add={}{\%}}%
	}},
	columns/13/.style={column name=Avg},
	columns/14/.style={column name=Max},
	columns/16/.style={column name=OOM,postproc cell content/.append code={
			\pgfkeysalso{@cell content/.add={}{\%}}%
	}},
	]{results.dat}
	
	\caption{Average and maximum memory usage (in GB) required by Algorithm~\ref{alg:neumannt} and 
	\ref{alg:neumann2}, AMEn, and DMRG for computing the MTTA of the case study in Section~\ref{sec:caseStudy}. The percentage of tests that ran out of memory is
	reported in the column OOM. The results are obtained by running $100$ random tests
	for each $k$, and taking average and maximum of the tests which did not encounter
	an OOM condition. This situation was never encounter for Algorithm~\ref{alg:neumann2}
	and \ref{alg:neumannt}. }
	\label{tab:memory}
\end{table}

\begin{table}
\centering 
\pgfplotstabletypeset[%
column type=c,
every head row/.style={
	before row={
		\toprule
		&\multicolumn{1}{c|}{}
		& \multicolumn{1}{c|}{Algorithm~\ref{alg:neumann2}} & 
		\multicolumn{1}{c|}{Algorithm~\ref{alg:neumannt}} & 
		\multicolumn{2}{c|}{AMEn} & 
		\multicolumn{2}{c}{DMRG} \\ 
	}, 
	after row = \midrule 
},
every last row/.style={after row=\bottomrule},		
sci zerofill,
columns={0,3,7,11,12,15,16},
columns/0/.style={dec sep align={c|},column type/.add={}{|},column name=$k$,fixed},
columns/3/.style={column name=Avg,column type/.add={}{|}},
columns/7/.style={column name=Avg,column type/.add={}{|}},
columns/11/.style={column name=Avg},
columns/12/.style={column name=OOM,column type/.add={}{|},postproc cell content/.append code={
		\pgfkeysalso{@cell content/.add={}{\%}}%
}},
columns/15/.style={column name=Avg},
columns/16/.style={column name=OOM,postproc cell content/.append code={
		\pgfkeysalso{@cell content/.add={}{\%}}%
}},
]{results.dat}
\caption{Average wall-clock time (in seconds) required to compute the MTTA for
the case study with $k$ components. OOM denotes the percentage of tests which were stopped
because they ran out of memory. Each run was allocated 20GB of RAM and 4 logical cores. The
average are computes on the tests which did not run out of memory. For Algorithm~\ref{alg:neumann2} and \ref{alg:neumannt}, this situation was never encountered.}
\label{tab:time}
\end{table}

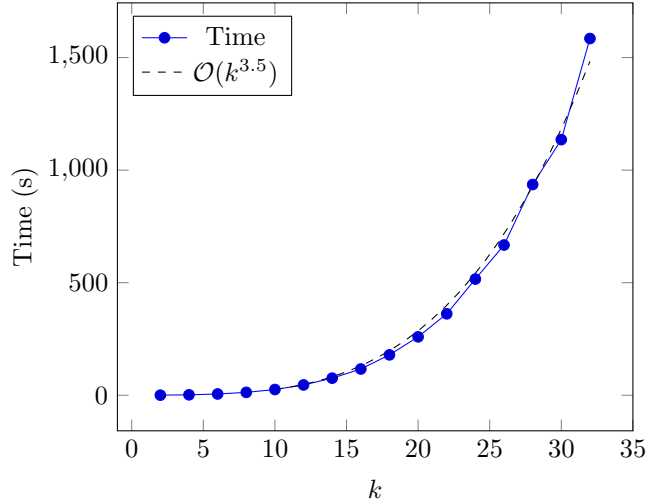
\begin{figure}	
	\centering
	\begin{tikzpicture}	
	\begin{axis}[xlabel = $k$, ylabel = Time (s), legend pos = north west]
	\addplot table[x index = 0, y index = 7] {results.dat};
	\addplot[domain = 10:32, dashed] {0.008 * sqrt(x^7)};
	\legend{Time, $\mathcal O(k^{3.5})$}
	\end{axis}
	\end{tikzpicture}
	\caption{Average timings as a function of $k$ for 
		Algorithm~\ref{alg:neumannt}. For this case study, the timings
	appear to depend on $k$ polynomially, with exponent close to $3.5$. }
	\label{fig:runtime}
\end{figure}

We have tested the values of $k = \{ 10, 12, \ldots, 32 \}$. For each value of $k$, we 
have run $100$ tests for Algorithm~\ref{alg:neumannt} and for Algorithm~\ref{alg:neumann2}, 
generating random topologies $\mathcal T$. The tests have been performed on a node of a cluster
with two Intel(R) Xeon(R) CPU E5-2650 v4 @ 2.20GHz processors each. The processes have
been limited to 20 GB of RAM and 4 threads each, with a time limit of 600 hours. 

The results for what concern  
memory usage are reported in \Cref{tab:memory}, and for runtime in \Cref{tab:time}. 
We note that, despite Algorithm~\ref{alg:neumann2} and \ref{alg:neumannt} being
equivalent (in the sense given in Section~\ref{sec:MTTA}), the quadratic convergence
of Algorithm~\ref{alg:neumann2} makes it the best choice on all the tests.

We note that Algorithm~\ref{alg:neumannt} has more predicable runtimes. In
Figure~\ref{fig:runtime}, it is visible that they appear to have a cubic dependency on
$k$ for the case study under consideration. Algorithm~\ref{alg:neumann2}, on the other
hand, has timings with a weaker correlation with $k$; from our observations, they
appear to be related to the topology, which influences the growth of the TT-ranks
during the iterations. This can be advantage, in the sense that even large scale cases
might be treatable, or a disadvantage, because it makes very hard to predict how long
the algorithm will need to give an answer. In particular, we have observed 
that for these problems the
choice of a good parameter $\gamma$ (as described in Section~\ref{sec:gamma}) is more important. In the case study under consideration, the dependency
of Algorithm~\ref{alg:neumannt} on $k$ appears to have an asymptotic
behavior close to $\mathcal O(k^{3.5})$, as reported in Figure~\ref{fig:runtime}.

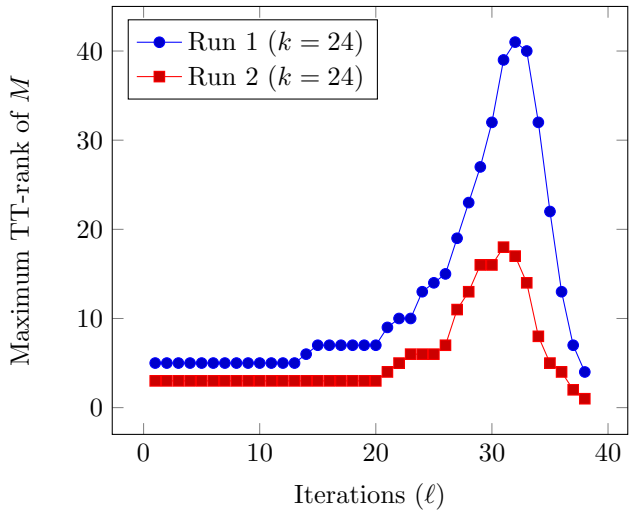
\begin{figure}	
	\centering
	\begin{tikzpicture}	
	\begin{axis}[xlabel = Iterations ($\ell$), ylabel = Maximum TT-rank of $M$, legend pos = north west]
	\addplot table[x index = 2, y index = 0] {ranks.dat};
	\addplot table[x index = 2, y index = 1] {ranks.dat};
	\legend{Run 1 ($k = 24$), Run 2 ($k = 24$)}
	\end{axis}
	\end{tikzpicture}
	\caption{Maximum TT-Rank of $M$ during the iteration of Algorithm~\ref{alg:neumann2} for 
	two runs with $k = 24$. The runs with smaller ranks took 4.55s, while the other
	needed 72.30s. }
	\label{fig:ranks}
\end{figure}

To show a typical behavior of the TT-ranks during the iteration of Algorithm~\ref{alg:neumann2}
we have reported two runs for $k = 24$, whose evolution of the maximum TT-rank of $M$
is reported in Figure~\ref{fig:ranks}. The two examples have been chosen one below and the 
other above the average runtime for this value of $k$. It is visible how the slowest 
of the two runs reaches a higher rank ($40$) than the other (which only gets up to $18$). 

We compared the results with the AMEn solver \cite{dolgov2012fast} and the
DMRG algorithm \cite{oseledets2012solution}, both available in the \texttt{TT-Toolbox} \cite{tttoolbox}. The AMEn solver has been proved to 
be quite effective for the computation of the steady-state vector of irreducible 
Markov chains in \cite{kressner2014low}. The solver can be used to compute 
$(Q - S)^{-1} v$, or to solve the normal equations $(Q - S)^T (Q - S) x = (Q - S)^T v$. 
The former problem is better conditioned, but the latter is symmetric positive definite, 
which guarantees convergence for the AMEn iteration. We have compared both choices, 
and we found that for this case study the second performs slightly better. However, 
the method stagnates on an increasing number of cases when $k > 10$, so we could only make a direct comparison in Table~\ref{tab:memory} and \ref{tab:time} for small values of $k$. DMRG, on the other hand, performed 
more favorably on the case study, and we have been able to solve problems (quite) 
reliably for $k$ 
up to $18$.

\section{Conclusions}
\label{sec:conclusions}

We have shown that tensor trains are a powerful tool for the analysis of 
performance and reliability measures (in this case, the mean time to
failure) of large systems, when the interconnection between the smaller
subsystems that compose them is sufficiently weak. 

We have presented a theoretical analysis of an iteration that is easily
applicable in the tensorized format, and with guaranteed convergence. 
A quadratically convergent variation has been shown as well, and the
performances have been tested on a representative set of examples. 

Several lines of research remain open: the connection between the
weak connections and the TT-rank in the iteration needs to be studied
further, in order to understand the relation more in depth. Moreover, 
several more measures are of interest in this context, and the application
of tensor techniques for this task could lead to faster and reliable
methods for their computation. 

We have introduced some techniques for accelerating the operations in tensor
formats for the iterations that arise from the Markovian context, and that will
be subject of future study. 

\section*{References}

\bibliographystyle{elsarticle-harv}
\bibliography{mttf}	

\end{document}